\RequirePackage{fix-cm}
\documentclass{svjour3}                     
\smartqed  
\usepackage{graphicx}
\usepackage{amssymb}
\usepackage{amsmath}
\usepackage{amsfonts}
\usepackage{epsfig}
\usepackage{psfig}
\usepackage[mathscr]{eucal}
\usepackage{setspace,color}
\usepackage{tikz}
\usetikzlibrary{snakes}
\usetikzlibrary{arrows,shapes}
\usepackage{esint}

                    {$\blacksquare$\vspace*{7pt}} 

\def\R{{\mathbb R}}

\def\N{{\mathbb N}}

\def\om{\omega}

\def\f{\frac}
\def\p{\partial}
\def\q{\quad}
\def\na{\nabla}

\def\nb{\nonumber}

\def\bi{{\mathbf i}}

\def\su{\mathrm{supp}}
\def\ssu{\mathrm{sing\text{-}supp}}

\def\CT{\mathrm{CT}}
\def\MA{\mathrm{MA}}

\def\x{{\mathbf x}}
\def\y{{\mathbf y}}
\def\z{{\mathbf z}}

\def\WF{\mathrm{WF}}

\def\TV{\mathrm{TV}}

\def\mC{{\mathcal C}}

\def\mI{{\mathcal I}}

\def\mQ{{\mathcal Q}}
\def\mR{{\mathcal R}}

\def\msD{{\mathscr D}}
\def\msE{{\mathscr E}}
\def\msF{{\mathscr F}}

\def\bi{\begin{itemize}} \def\ei{\end{itemize}}
\def\be{\begin{eqnarray*}}
\def\ee{\end{eqnarray*}}

\def\etal{{\it et al} \:}

\def\0{{\mathbf 0}}

\newcommand{\beq}{\begin{equation}}
\newcommand{\eeq}{\end{equation}}

\def\xxi{{\boldsymbol\xi}}
\def\eet{{\boldsymbol\eta}}

\def\bth{\boldsymbol{\theta}}

\def\eref#1{\eqref{#1}}

\newcommand{\eps}{\varepsilon}

\def\XXint#1#2#3{{\setbox0=\hbox{$#1{#2#3}{\int}$ }
\vcenter{\hbox{$#2#3$ }}\kern-.55\wd0}}

\begin{document}

\title{Characterization of Metal Artifacts in X-ray Computed Tomography
}
\subtitle{}

\titlerunning{Characterization of metal artifacts in CT}        

\author{Hyoung Suk Park \and Jae Kyu Choi \and Jin Keun Seo
}

\authorrunning{H.S.Park \and J.K.Choi \and J.K.Seo} 

\institute{Hyoung Suk Park \and Jae Kyu Choi \and Jin Keun Seo \at
              Department of Computational Science and Engineering, Yonsei University, 120-749 Korea \\
              Tel.: 82-2-2123-6121\\
              Fax: 82-2-2123-8194\\
              \email{jiro7733@yonsei.ac.kr \and jaycjk@yonsei.ac.kr \and seoj@yonsei.ac.kr}           
}

\date{Received: date / Accepted: date}

\maketitle

\begin{abstract} Metal streak artifacts in X-ray computerized tomography (CT) are rigorously characterized here using the notion of the wavefront set from microlocal analysis. The metal artifacts are caused mainly from the mismatch of the forward model of the filtered back-projection; the presence of metallic subjects in an imaging subject violates the model's assumption of the CT sinogram data being the Radon transform of an image. The increasing use of metallic implants has increased demand for the reduction of metal artifacts in the field of dental and medical radiography. However, it is a challenging issue due to the serious difficulties in analyzing the X-ray data, which depends nonlinearly on the distribution of the metallic subject.  In this paper, we found that the metal streaking artifacts cause mainly from the boundary geometry of the metal region. The metal streaking artifacts are produced only when the wavefront set of the Radon transform of the characteristic function of a metal region does not contain the wavefront set of the square of the Radon transform.  We also found a sufficient condition for the non-existence of the metal streak artifacts.
\end{abstract}

\newpage

\section{Introduction}

X-ray computed tomography (CT) is one of the most powerful diagnostic tools for medical and dental imaging. It provides tomographic images of the human body by assigning an X-ray attenuation coefficient to each pixel \cite{Tohnak2007}. However, patients with metal implants may not receive the benefits of CT scanning because the quality of the image can be greatly degraded by metal streaking artifacts that appear as dark and bright streaks. Medical implants such as coronary stents, orthopedic implants, surgical clips, and dental fillings disturb the accurate visualization of anatomical structures, rendering the images useless for diagnosis. Therefore, various research efforts have sought to develop metal artifact reduction methods, but this goal remains one of the major challenges facing CT imaging.

This paper aims, through rigorous mathematical analysis, to characterize the structure of metal streaking artifacts.  The artifacts arise because of the mismatch between the CT image reconstruction algorithm (filtered back projection (FBP) algorithm \cite{Bracewell1967}) and the nonlinear variation in the X-ray data that occurs in the presence of a metallic object.
In CT, X-ray projection data $P(\varphi,s)$ are collected for a slice after passing X-ray beams in different directions through an object, where $s$ indicates the position of the projected line and $\varphi$ is the angle of the projection as shown in Fig. \ref{fig:basic_CT}.  The FBP algorithm is based on the assumption that the  X-ray projection data $P(\varphi,s)$ is in the range of the Radon transform \cite{Radon2005},  with its domain being  $\msE'(\R^2)$, the space of distributions whose supports are compact; accepting this assumption, there exists $f$ in $\msE'(\R^2)$ such that
\begin{align}\label{Rf00}
P(\varphi,s)=\mR f(\varphi,s)\q\mbox{for  } (\varphi,s)\in(-\pi,\pi]\times\R
\end{align}
where $\mR f(\varphi,s):=\int_{\R^2}f(\x)\delta(\x\cdot\bth-s)d\x$, $\bth=(\cos\varphi,\sin\varphi)$, and $\x=(x_1,x_2)\in\R^2$ \cite{Natterer1986}.   Then, the tomographic image $f$  can be reconstructed by the following FBP formula:
\begin{align}\label{FBP_formula}
f=\f{1}{4\pi}\mR^* \mI^{{\tiny\mbox{$-1$}}}\mR f\q\q \mbox{for }~f\in \msE'(\R^2)
\end{align}
where $\mR^*$ is the adjoint of Radon transform given by $\mR^* h(\x)=\int_{-\pi}^{\pi} h(\varphi,\x\cdot\bth)d\varphi$, and $\mI^{{\tiny\mbox{$-1$}}}g$ is the Riesz potential given by  $\mI^{{\tiny\mbox{$-1$}}}g(s)=\f{1}{2\pi}\int_\R\int_\R e^{i(s'-s)\om} g(s') |\om| ds'd\om $. FBP  works well for CT imaging  because for most human tissues, the X-ray data $P(\varphi,s)$ approximately  satisfy the linear assumption of \eref{Rf00}.

\begin{figure}[ht]
\begin{center}
\includegraphics[width=12cm]{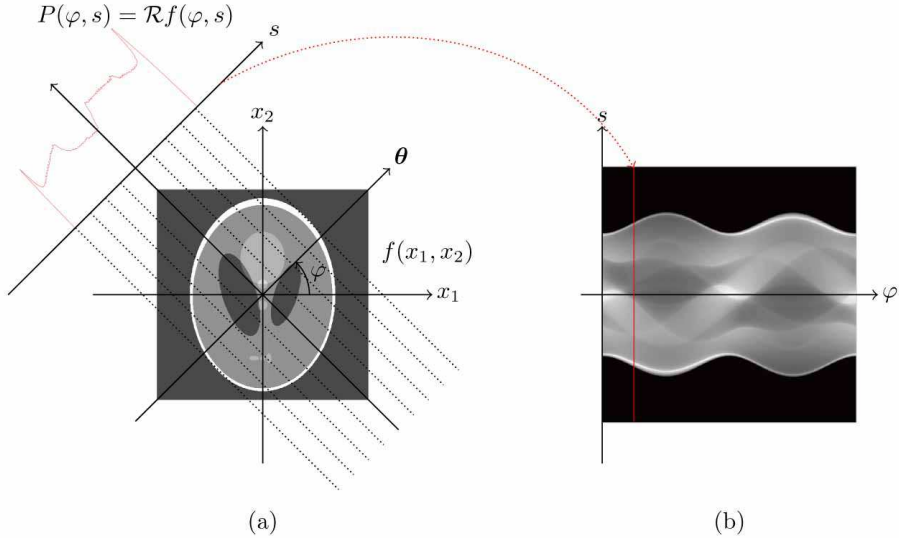}
\caption{(a) Schematic diagram of CT data acquisition, (b) Resulting sinogram $P(\varphi,s)$.}\label{fig:basic_CT}
\end{center}
\end{figure}

However,  metallic objects in the imaging slice cause the X-ray data $P(\varphi,s)$ to fail to satisfy the assumption of \eref{Rf00}. This is because the incident X-ray beams comprise a number of photons of different energies ranging between $\underline{E}$ and $\overline{E}$, and  the X-ray attenuation coefficients vary with $E$. We use the notation $f_E(\x)$ to describe the attenuation coefficient at position $\x$ and at  energy level $E$. Assuming that there is no scattered radiation or noise, the data $P(\varphi,s)$ can be expressed by Beer's law \cite{Beer1852,Lambert1892}:
  \begin{align}\label{P_theta}
P(\varphi,s)=-\ln\left(\int_{\underline{E}}^{\overline{E}}\eta(E)\exp\big\{-\mR f_E(\varphi,s)\big\}dE\right)
\end{align}
where $\eta$ is a probability density function supported on $[\underline{E},\overline{E}]$. Hence, the assumption of \eref{Rf00} may not hold true when  $f_E$ varies  with $E$. In particular, the attenuation coefficient of a metal object varies greatly with $E$, and hence  the presence of metal objects in the scan slice  leads  the data $P(\varphi,s)$ to violate the assumption of \eref{Rf00}.

Assuming $f_E$ is differentiable with respect to $E\in [\underline{E},\overline{E}]$, the mismatch between the data $P(\varphi,s)$ and $\mR f_{E_0}$ for an effective energy $E_0$, which is called the beam hardening effect, can be analyzed by the quantity \cite{Seo2012}
\begin{equation}\label{mismatch}
[P-\mR f_{E_0}](\varphi,s)\approx\int_{\underline{E}}^{\overline{E}}\eta(E')\int_{E_0}^{E'}\left[ e^{-\mathcal R[f_{E}-f_{E_0}](\varphi,s)}\mathcal R\left[\f{\p f_{E}}{\p E}\right](\varphi,s)\right]dE dE'
\end{equation}
where $E_0$ is a mean energy level $E_0\in [\underline{E},\overline{E}]$. Hence, the condition of $P-\mR f_{E_0}=0$ (perfect matching) requires that we have either $\underline{E}-\overline{E}=0$ (monochromatic X-rays.  Otherwise, we say that X-rays are polychromatic.) or $\f{\p f_{E}}{\p E}=0$. The absolute value of $\f{\p f_{E}}{\p E}$ is large for metallic subject (for instance, $\f{\p f_E}{\p E}\approx -1650\mathrm{cm^2/g\cdot MeV}$ on $[0.01\mathrm{MeV},0.08\mathrm{MeV}]$ for gold \cite{J.H.Hubbell1995}), and it is difficult to generate monochromatic X rays for routine clinical use \cite{Dilmanian1997,Natterer2002}.

To carry out rigorous analysis for metal artifacts, let the metal region in the imaging slice occupy the domain $D$. Since $f_E \approx f_{E_0}$ in most tissues and $|\f{\p f_E}{\p E}|$ is large in $D$, we can approximate $f_E$ by
\begin{align}\label{f_E_simple}
 f_E(\x)\approx f_{E_0}(\x) + \alpha(E-E_0)\chi_D(\x)
\end{align}
where $\chi_D$ is the characteristic function of $D$ and $\alpha=\f{\p f_E}{\p E}$ in the metal $D$.

Based on the linearized assumption \eref{f_E_simple}, we characterize the metal streaking artifacts using the framework of the wavefront set \cite{Hormander1983,Petersen1983,Tr`eves1980}. The mismatch of the projection data is expressed by
\begin{align*}
-\f{1}{4\pi} \mR^* \mI^{{\tiny\mbox{$-1$}}} \left[\ln\left(\f{\sinh\left( \f{\alpha}{2}(\overline{E}-\underline{E})  \mR\chi_D\right)}{\f{\alpha}{2}(\overline{E}-\underline{E}) \mR\chi_D}\right)\right].
\end{align*}
We found that the metal streaking artifacts are produced only when
\begin{align*}
\WF[(\mR\chi_D)^2]\nsubseteq  \WF[\mR\chi_D]
\end{align*}
where $\WF(g)$ is the wavefront set of a function $g$ defined on $(-\pi,\pi]\times \R$. (See Theorem \ref{th32}.) We also find the necessary condition for the existence of the metal artifacts; the reconstructed image contains streaking artifacts only if $D$ is not strictly convex. (See Theorem \ref{th31}.)  Using a similar argument as in the proof of Theorem \ref{th32}, we present characterizations of other effects that cause  streaking artifacts, such as scattered radiations and  noises (See Section \ref{othereffects}.)  Finally, we provide numerical simulation results and clinical CT image to support these observations (See Fig. \ref{fig:streaking_illustration}--\ref{real_image}.)

Various works have studied the wavefront set for Radon transforms \cite{DeHoop2009,Finch2003,J.Frikel2013,Greenleaf1989,Katsevich1999,Katsevich2006,Quinto1980,Quinto1993,Quinto2006,Quinto2007,Quinto2013,Ramm1996,Ramm1993} and metal artifacts \cite{Abdoli2010,Bal2006,J.Choi2011,DeMan2001,Kalender1987,Lewitt1978,Meyer2010,Park2013,Shen2002,Wang1996,Zhao2000}, but surprisingly this paper reports as far as we know the first rigorous mathematical analysis to characterize the structure of metal streaking artifacts.

\section{Mathematical Framework}

Before providing the main results on the metal streaking artifacts, we begin with a brief summary of the basic mathematical principles on X-ray CT. Let $f_E$ denote the attenuation coefficient distribution of the slice of an object being imaged at the energy level $E$. When the X-ray pass through the slice along the direction $\bth=(\cos\varphi,\sin\varphi)$, the X-ray data $P(\varphi,s)$ is given by
\begin{align}\label{P_theta1}
P(\varphi,s):=-\ln\left[\int_{E_0-\delta}^{E_0+\delta}\eta(E)\exp\big\{-\mR f_E(\varphi,s)\big\}dE\right]
\end{align}
where $\delta =\f{\overline{E}-\underline{E}}{2}$.

We denote by $f_{\CT}$ the reconstructed CT image obtained by FBP \eref{FBP_formula}:
\begin{align}\label{f_ct}
f_{\CT}(\x)=\f{1}{4\pi}\mR^*[\mI^{{\tiny\mbox{$-1$}}}P](\x)=\f{1}{8\pi^2}\int_{-\pi}^{\pi}\int_{-\infty}^{\infty}|\om|\msF[P(\varphi,\cdot)](\om)e^{i\om\x\cdot\bth}d\om d\varphi
\end{align}
If $\f{\p f_E}{\p E} (\x)=0$ ($f_E$ is independent to $E$), then $P(\varphi,s)=\mR f_{E_0}(\varphi,s)$, and therefore FBP formula \eref{f_ct} gives
\begin{align}\label{fct:no_artifact}
f_{\CT}=f_{E_0}.
\end{align}
However, the above identity \eref{fct:no_artifact} fails when $\f{\p f_E}{\p E}\neq0$.

Assuming that $f_E$ is twice differentiable with respect to $E\in[E_0-\delta,E_0+\delta]$, $f_E(\x)$ can be expressed as
\begin{align}\label{f_E1}
f_E(\x)=f_{E_0}(\x)+(E-E_0)\f{\p f_E}{\p E}(\x)\bigg|_{E=E_0}+O(|E-E_0|^2).
\end{align}
With a properly chosen energy window, $f_E\approx f_{E_0}$ for most human tissues. However, the magnitude of $\f{\p f_E}{\p E}$ is large for metallic materials.  Hence, we assume that
\begin{align}\label{alpha_j}
\f{\p f_{E}}{\p E}\bigg|_{E=E_0}(\x)=\left\{\begin{array}{cl}
0~&\mbox{if}~\x\notin D_j\\
\alpha_j\neq0~&\mbox{if}~\x\in D_j,
\end{array}\right.
\end{align}
where $D_1,\cdots,D_N$ denote subregions of a metal region $D\subseteq\R^2$ and $\alpha_1,\cdots,\alpha_N<0$ are constants depending on the metallic materials.  Noting that $f_E(\x)$ approximately satisfies a linear relation with respect to $E$ on the practical energy window level $[E_0-\delta,E_0+\delta]$ \cite{J.H.Hubbell1995}, we assume that $f_E(\x)$ satisfies
\begin{align}\label{f_E}
f_E(\x)=f_{E_0}(\x)+(E-E_0)\sum_{j=1}^N\alpha_j\chi_{D_j}(\x).
\end{align}
Here, $\chi_{D_j}$ denotes the characteristic function of $D_j$; $\chi_{D_j}=1$ in $D_j$ and $0$ otherwise.

For simplicity's sake, we assume $\eta=1/2\delta$ and $\alpha_1=\cdots=\alpha_N=\alpha$, which simplify the expression of  $f_E(\x)$ and $P(\varphi,s)$:
\begin{align}
\label{f_E_simple1} f_E(\x)&=f_{E_0}(\x)+\alpha(E-E_0)\chi_D(\x)\\
\label{P_theta_simple} P(\varphi,s)&=-\ln\left(\f{1}{2\delta}\int_{E_{0}-\delta}^{E_0+\delta}\exp\{-\mR f_{E_0}(\varphi,s)-\alpha (E-E_0)\mR \chi_D(\varphi,s)\}dE\right)
\end{align}

In order to explain the metal artifacts viewing as the singularities in an image, we need to choose proper spaces to contain $f_E(\x)$ and the projection data $P(\varphi,s)$. Let $C_0^{\infty}(\R^2)$ denote the space of smooth and compactly supported functions on $\R^2$. Let $\msD'(\R^2)$ denote the space of distributions, continuous linear functionals on $C_0^{\infty}(\R^2)$.  For $u\in\msD'(\R^2)$, its support, denoted as $\su(u)$, is the smallest closed subset of $\R^2$ outside of which $u$ vanishes. Throughout this paper, we assume:
\begin{enumerate}
\item[A1.] $f_E\in \msE'(\R^2)$, where $\msE'(\R^2)$ is the space of distributions of compact support.
\item[A2.] $P\in\msE'((-\pi,\pi]\times\R)$, where  $\msE'((-\pi,\pi]\times\R)$ is the space of distributions on $(-\pi,\pi]\times\R$ which are compactly supported with respect to the second variable.
\item[A3.] In the metal region $D$, $f_{E_0}$ satisfies
\begin{align*}
f_{E_0}(\x)\geq C\sup_{\x\in\overline{D}^c}f_{E_0}(\x)~~~~~~\mbox{for some }~C>1.
\end{align*}
\end{enumerate}

The following proposition expresses the decomposition of the filtered backprojected CT image $f_{\CT}$ into the metal artifact-free term ($f_{E_0}$) and the metal artifact term ($f_{\MA}$).
\begin{proposition}\label{proposition22} The  $f_{\CT}$ in \eref{f_ct} can be decomposed into
\begin{align}\label{fct}
f_{\CT}(\x)=f_{E_0}(\x)+f_{\MA}(\x)\end{align}
where $f_{\MA}$ represents the metal artifact term given by
\begin{align}\label{fma}
f_{\MA}(\x)=-\f{1}{8\pi^2}\int_{-\pi}^{\pi}\int_{-\infty}^\infty  |\om|\msF\left[\ln\left(\f{\sinh\left( \alpha\delta\mR \chi_D(\varphi,\cdot)\right)}{\alpha\delta\mR \chi_D(\varphi,\cdot)}\right)\right](\om)e^{i\omega \x\cdot\bth}d\om d\varphi.
\end{align}
\end{proposition}

\begin{proof}~
The direct computation of \eref{P_theta_simple} yields
\begin{align}\label{P_theta_simple2}
P(\varphi,s)&=\mR f_{E_0}(\varphi,s)-\alpha E_0\hspace{0.1em}\mR \chi_D (\varphi,s)-\ln\left(\f{1}{2\delta}\int_{E_{0}-\delta}^{E_0+\delta}\exp\{-\alpha E\hspace{0.1em}\mR \chi_D (\varphi,s)\}dE\right)\nb\\
&=\mR f_{E_0}(\varphi,s)-\ln\left(\f{\sinh(\alpha\delta\mR \chi_D (\varphi,s))}{\alpha\delta\mR \chi_D (\varphi,s)}\right).
\end{align}
Substituting \eref{P_theta_simple2} into \eref{f_ct} leads to \eref{fma}.  \qed
\end{proof}
According to Proposition \ref{proposition22}, the CT image $f_{\CT}$ in \eref{fct} is nonlinear with respect to the geometry of the metal region $D$. This nonlinear property of $f_{\CT}$ with respect to the geometry of metallic objects in the field of view is related with streaking artifacts in $f_{\CT}$, which will be explained in the following section in detail.

\section{Main Results}
This section provides a rigorous analysis of the metal artifacts.   We knew that metal artifacts are mainly caused by the large variation in the attenuation coefficients of the metals with respect to the energy level, which causes a significant distance between the projection data $P$ and the range space $\mR(\msE'(\R^2))$.  Metal streaking artifacts are closely related to the interrelation between the structure of the data $P$ and the FBP. This relation can be interpreted effectively using the Fourier integral operator and the wave front set \cite{J.J.Duistermaat1972,Hormander1971,Tr`eves1980}.

Note that for each $E\in[E_0-\delta,E_0+\delta]$, the attenuation coefficient $f_E$ is bounded and compactly supported.  Similarly, we can note that $P(\varphi,s)$ is compactly supported with respect to $s$ variable for each $\varphi\in(-\pi,\pi]$.  Hence, we can say that the spaces $\msE'(\R^2)$ and $\msE'((-\pi,\pi]\times\R)$ contain all meaningful attenuation coefficient distributions on each energy level $E$ and all practical projection data.  In addition, using the duality
\begin{align}
\int_{-\pi}^{\pi}\int_{-\infty}^{\infty}\mR f(\varphi,s)g(\varphi,s)dsd\varphi=\int_{\R^2}f(\x)\bigg[\underbrace{\int_{-\pi}^{\pi}g(\varphi,\x\cdot\bth)d\varphi}_{\mR^*g(\x)}\bigg]d\x,
\end{align}
Radon transform $\mR$ can be extended to a weakly continuous map from $\msE'(\R^2)$ to $\msE'((-\pi,\pi]\times\R)$ \cite{J.Frikel2013}.

The wavefront set is a useful tool to describe simultaneously the locations  and orientations of singularities.  First of all, $V\subseteq\R^2\setminus\{\0\}$ is called a conic set if $r\xxi\in V$ whenever $r>0$ and $\xxi\in V$.  If $V$ is an open conic set which contains $\xxi\neq\0$, we say that it is a conic neighborhood of $\xxi$.  If $U\subseteq\R^2$ and $V\subseteq\R^2\setminus\{\0\}$ is a conic set, then so is $U\times V\subseteq\R^2\times(\R^2\setminus\{\0\})$, and we say that $U\times V$ is conically compact if $U$ is compact and $V$ is conic.

\begin{definition}
Let $u\in\msD'(\R^2)$, $v\in\msE'(\R^2)$, and $\x\in\R^2$.
\begin{enumerate}
\item The singular support of $u$, denoted as $\ssu(u)$, is the smallest closed subset in $\R^2$ outside of which $u$ is $C^{\infty}$.
\item $\Sigma(v)$ is the smallest closed conic subset of $\R^2\setminus\{\0\}$ outside of which $\msF[v]$ decays rapidly. In other words, if $\xxi\notin\Sigma(v)$, then there is a conic neighborhood $V$ of $\xxi$ such that
    \begin{align*}
    \sup_{\xxi'\in V}(1+|\xxi'|)^N|\msF[v](\xxi')|<\infty~~~~~\mbox{for every }~N\in\N.
    \end{align*}
\item For $\x\in\R^2$, $\Sigma_{\x}(u)$ is a closed conic subset in $\R^2\setminus\{\0\}$ defined as
\begin{align*}
\Sigma_{\x}(u)=\bigcap\{\Sigma(\eta u):\eta\in C_c^{\infty}(\R^2),~\eta(\x)\neq0\big\}.
\end{align*}
\item The wavefront set of $u$, denoted as $\WF(u)$, is a closed conic subset in $\R^2\times(\R^2\setminus\{\0\})$ defined as
\begin{align*}
\WF(u)=\big\{(\x,\xxi)\in\R^2\times(\R^2\setminus\{\0\}):\xxi\in\Sigma_\x(u)\big\}.
\end{align*}
\end{enumerate}
\end{definition}
We can note that for $v\in\msE'(\R^2)$, $v\in C_0^{\infty}(\R^2)$ if and only if $\Sigma(v)=\emptyset$.  This means that
\begin{align*}
\ssu(u)=\big\{\x\in\R^2:\Sigma_{\x}(u)\neq\emptyset\big\}~~~~\mbox{for}~~~u\in\msD'(\R^2).
\end{align*}
If $(\x,\xxi)\in\WF(u)$, then $\xxi\in\Sigma(u)$ for $u\in\msE'(\R^2)$ \cite{Hormander1983}.

\begin{definition}Let $f_{\CT}$ be  a function  in \eref{fct}. A straight line $L_{\varphi,s}=\big\{\x=s(\cos\varphi,\sin\varphi)+t(-\sin\varphi,\cos\varphi): t\in \R\big\}$ is called a \emph{streaking artifact of} $f_\CT$  in the sense of wavefront set if it satisfies
\begin{align}\label{artifact0}
\Sigma_\x(f_{\CT})\neq \emptyset \q\mbox{for all }~ \x\in L_{\varphi,s}.
\end{align}
\end{definition}
To investigate $\WF(f_{\CT})$, we express the Radon transform $\mR$ of $f_{E_0}\in\msE'(\R^2)$ in the form of Fourier integral operator (FIO) \cite{J.J.Duistermaat1972,Hormander1971,Tr`eves1980}
\begin{align}\label{Rf}
\mR f_{E_0}(\varphi,s)&=\f{1}{2\pi}\int_{-\infty}^{\infty}\int_{\R^2}f_{E_0}(\x)e^{i\phi(\x,(\varphi,s),\om)}d\x d\om
\end{align}
with $\phi(\x,(\varphi,s),\om)=\om(s-\x\cdot\bth)$. Then, the canonical relation (wavefront set of the kernel of FIO) \cite{Guillemin1990,Quinto2006,Tr`eves1980} is given by \begin{align*}
\mC_{\mR}=\{((\varphi,s),a(-\x\cdot\bth^{\perp},1);\x,a\bth):a\neq 0~\&~\x\cdot\bth=s\},
\end{align*}
and the standard argument of the wavefront set in \cite{Hormander1983,Tr`eves1980} yields
\begin{align}\label{WFRadon}
\WF(\mR f_{E_0})&\subseteq\mC_{\mR}\circ\WF(f_{E_0})\end{align}
where
\begin{align*}
\mC_{\mR}\circ\WF(f_{E_0}):=\big\{((\varphi,s),a(-\x\cdot\bth^{\perp},1)):a\neq0,~\exists(\x,a\bth)\in\WF(f_{E_0})~\mbox{s.t.}~\x\cdot\bth=s\big\}.
\end{align*}
Similarly, we can write the backprojection $\mR^*g$ of $g\in\msE'((-\pi,\pi]\times\R)$ in the form of FIO
\begin{align*}
\mR^*g(\x)&=\int_{-\pi}^{\pi}g(\varphi,\x\cdot\bth)d\varphi=\f{1}{2\pi}\int_{-\infty}^{\infty}\int_{-\infty}^{\infty}\int_{-\pi}^{\pi}g(\varphi,s)e^{i\om(\x\cdot\bth-s)}d\varphi dsd\om,
\end{align*}
with the canonical relation $\mC_{\mR^*}$:
\begin{align*}
\mC_{\mR^*}=\{(\x,a\bth;(\varphi,s),a(-\x\cdot\bth^{\perp},1)):a\neq 0~\&~\x\cdot\bth=s\}.
\end{align*}
Then the wave front set of $\mR^*g$ satisfies
\begin{align}\label{WFbackprojection}
\WF(\mR^*g)&\subseteq\mC_{\mR^*}\circ\WF(g)\\
\nb &=\{(\x,a\bth):a\neq0,~\exists ((\varphi,s),a(-\x\cdot\bth^{\perp},1))\in\WF(g)~\text{s.t.}~\x\cdot\bth=s\}
\end{align}
Here, we extend $g\in\msE'((-\pi,\pi)\times\R)$ periodically with respect to the first variable $\varphi$ and choose $\psi\in C_0^{\infty}(\R^2)$ with $\su(\psi)\subseteq(-\pi,\pi]\times\R$.  By doing so, we can treat $\psi g$ as an element of $\msE'(\R^2)$ and find $\WF(g)$ \cite{Quinto2006}.

Now, we are ready to explain the main theorem which provides the characterization of  the metal artifacts in term of geometry of the metallic objects.
\begin{theorem}\label{th32} Let $\chi_D$ denote the characteristic function of the metal region $D\subseteq\R^2$. Let $f_{\CT}$ be the function in \eref{fct}. Then the necessary condition for existence of streaking artifacts of  $f_{\CT}$
is
\begin{align}\label{inclusionWF0}
\WF((\mR \chi_D)^2)\nsubseteq\WF(\mR \chi_D).
\end{align}
Moreover, if a line $L_{\varphi,s}$ is a streaking artifact of $f_\CT$ in the sense of wavefront set \eref{artifact0}, then $(\varphi,s)$ satisfies
\begin{align}\label{streaking3}
\dim\left(\mathrm{Span} [ \Sigma_{(\varphi,s)}(\mR\chi_D)]\right)=2
\end{align}
where $\dim(\mathrm{Span} [A])$ is the dimension of the span of the set $A$.
\end{theorem}

Before proving the theorem, let us understand its meaning.
If $(\varphi,s)$ satisfies \eref{streaking3}, then there exist $ t_1, t_2 \in \R$ such that
\begin{align}\label{streaking4}
t_1\neq t_2 ~~\mbox{and}~~(-t_1,1),(-t_2,1)\in\Sigma_{(\varphi,s)}(\mR\chi_D).
\end{align}
Using the fact that there is a one-to-one correspondence \cite{Quinto2006} between $\WF(f_{E_0})$ and $\WF(\mR f_{E_0})$, \eref{streaking4} gives the existence of two distinct points $\y$, $\z\in\p D$  such that   \begin{align*}
(\y,\bth), (\z,\bth)\in \WF(\chi_D)
\end{align*}
where  $\bth=(\cos\varphi,\sin\varphi)$ with $\varphi$ being the angle in \eref{streaking4} \cite{Hormander1983,Oneill2006,Quinto2006}. Then $L_{\varphi,s}$ is the straight line containing two points $\y$ and $\z$ as shown in Fig. \ref{fig:streaking_illustration2}.
\begin{figure}[ht]
\begin{center}
		\includegraphics[width=1\textwidth]{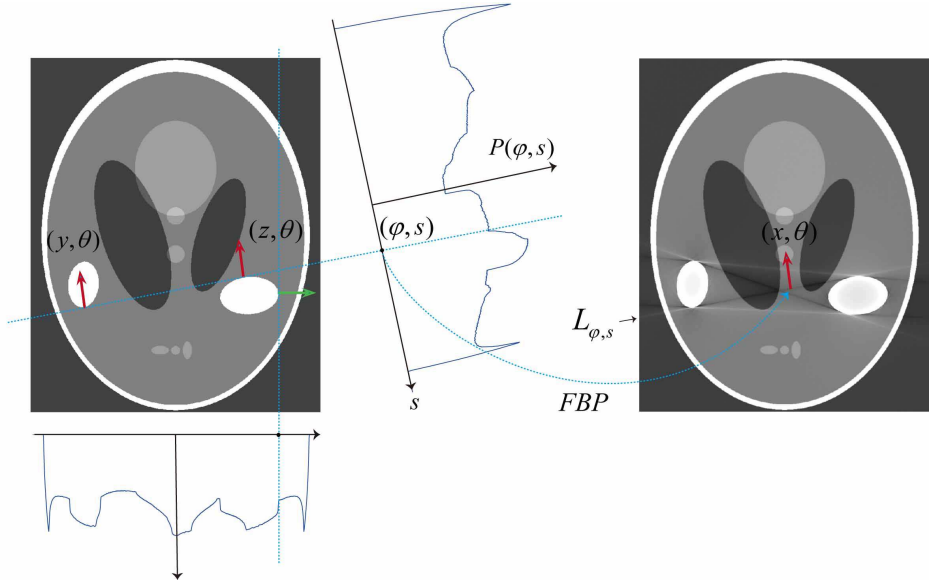}
		\caption{Illustration of streaking artifacts in the presence of metallic objects. Left figure shows the Shepp-Logan phantom with homogeneous metallic objects and the corresponding reconstructed  $f_{\CT}$ (right figure) with the display window [-0.02 0.04], respectively.}\label{fig:streaking_illustration2}
\end{center}
\end{figure}

\begin{proof} According to Proposition \ref{proposition22}, the wavefront set of $f_{\CT}$ satisfies
\begin{align}\label{WFfct}
\WF(f_{\CT})\subseteq\WF(f_{E_0})\cup\WF(f_{\MA})
\end{align}
and $f_{\MA}$ has the following expansion:
\begin{align}\label{fma2}
f_{\MA}= \f{1}{4\pi}\sum_{k=1}^\infty\f{(-1)^{k} }{k}\left[\sum_{n=1}^{\infty}\f{(\alpha\delta)^{2n}}{(2n+1)!}\mR^*\mI^{{\tiny\mbox{$-1$}}}(\mR \chi_D)^{2n}\right]^k.
\end{align}
Hence,
\begin{align}\label{fma3}
\WF(f_{\MA})\subseteq  \bigcup_{k=1}^\infty\WF\left(\mR^* \mI^{{\tiny\mbox{$-1$}}}(\mR \chi_D)^{2k}\right)
\end{align}
Since $\mI^{{\tiny\mbox{$-1$}}}$ is an elliptic pseudodifferential operator \cite{Petersen1983,Tr`eves1980},
\begin{align}\label{fma4}
\WF(f_{\MA})\subseteq \bigcup_{k=1}^\infty\WF\left(\mR^* (\mR \chi_D)^{2k}\right)
\end{align}

The wavefront set $\WF(\mR\chi_D)$ can be decomposed into
\begin{align*}
\WF(\mR\chi_D)=\WF_1(\mR\chi_D)\cup\WF_2(\mR\chi_D)
\end{align*}
where
\begin{align*}
\WF_k(\mR\chi_D)=\big\{((\varphi,s),\eet)\in\WF(\mR\chi_D):\dim(\mathrm{Span}[\Sigma_{(\varphi,s)}(\mR\chi_D)])=k\big\},~k=1,2.
\end{align*}

{\it Case 1}. Assume $\WF_2(\mR\chi_D)=\emptyset$. Since $\WF(\mR\chi_D)=\mC_{\mR}\circ\WF(\chi_D)$  due to the Bolker assumption \cite{Guillemin1990,Quinto2006},  for each $(\varphi,s)\in\ssu(\mR\chi_D)$, there exists the unique $t\in\R$ such that
\begin{align*}
((\varphi,s),(-t,1))\in\WF(\mR\chi_D).
\end{align*}
To analyze $\WF((\mR\chi_D)^2)$, we use the following fact on the product of distributions $u\in\msD'((-\pi,\pi]\times\R)$ \cite{Hormander1983,Jin2012}:
\begin{align}\label{WFproduct}
\WF(u^2)=\WF(u)\cup\big\{(\y,\eet_1+\eet_2):(\y,\eet_1),~(\y,\eet_2)\in\WF(u)~\&~\eet_1+\eet_2\neq\0\big\}
\end{align}
where $\y=(\varphi,s)$.
From \eref{WFproduct}, we have
\begin{align}\label{WFproduct2}
\Sigma_{(\varphi,s)}((\mR\chi_D)^2)=\Sigma_{(\varphi,s)}(\mR\chi_D)\subseteq\Sigma_{(\varphi,s)}(\mR f_{E_0}),
\end{align}
In general, we have
\begin{align}\label{WFproduct3}
\Sigma_{(\varphi,s)}((\mR\chi_D)^k)\subseteq\Sigma_{(\varphi,s)}(\mR f_{E_0})~~~~~\mbox{for } k=2,3,\cdots.
\end{align}
Hence, it follows from \eref{WFbackprojection} and \eref{WFproduct3} that
\begin{align*}
\WF(f_{\MA})\subseteq\bigcup_{k=1}^{\infty}\WF\left(\mR^* (\mR \chi_D)^{2k}\right)\subseteq\WF(\chi_D).
\end{align*}
This means that $\WF(f_{\MA})\subseteq \WF(f_{E_0})$  in the case when $\WF_2(\mR\chi_D)=\emptyset$. Hence, $f_{\CT}$ does not have streaking artifacts in the sense of wavefront set.

Next, we consider the remaining case.

{\it Case 2.} Assume $\WF_2(\mR\chi_D)\neq\emptyset$. Then, there exist $(\varphi,s)\in\ssu(\mR\chi_D)$ such that   $((\varphi,s),\eet)\in\WF_2(\mR\chi_D)$.  In other words, there exist $t_1$, $t_2\in\R$ such that
\begin{align*}
t_1\neq t_2~\&~((\varphi,s),(-t_1,1)),~((\varphi,s),(-t_2,1))\in\WF(\mR\chi_D).
\end{align*}
Then, it follows from the product formula \eref{WFproduct} that
\begin{align}\label{line0}
\Sigma_{(\varphi,s)}((\mR\chi_D)^2)=\big\{a(-t,1):a\neq0~\&~t\in\R\big\},
\end{align}
which means that
\begin{align}\label{line2}
\Sigma_{(\varphi,s)}((\mR\chi_D)^2)\nsubseteq\Sigma_{(\varphi,s)}(\mR\chi_D).
\end{align}
Moreover, \eref{line0} gives
\begin{align*}
\WF((\mR\chi_D)^2)\ni\left((\varphi,s),\f{1}{|t|}(-t,1)\right)\longrightarrow((\varphi,s),(\mp1,0))
~~\mbox{as } t\to\infty.\end{align*}
Hence, $((\varphi,s),(\mp1,0))$ lies in the limit (or asymptotic) cone of $\WF((\mR\chi_D)^2)$, and $((\varphi,s),(-t,1))\in\WF((\mR\chi_D)^2)$ for every $t\in\R$. Therefore, $\WF(\mR^*(\mR\chi_D)^2)$ will have a singularity which propagates along the straight line $L_{\varphi,s}$ by \eref{WFbackprojection}.

From Case 1 and Case 2, we obtain
\begin{align*}
\WF\left((\mR \chi_D)^{2}\right)\subseteq\WF(\chi_D)\q\mbox{ if and only if }\q \WF_2(\mR\chi_D)=\emptyset.\end{align*}
Moreover, if $L_{\varphi,s} \subseteq \ssu(f_\CT)$ (a streaking artifact), then it must be
\begin{align*}((\varphi,s),(-t,1)) \in\bigcup_{k=1}^{\infty}\WF\left((\mR \chi_D)^{2k}\right)\q\mbox{  for all } t\in\R,\end{align*}
 which is possible only when $\dim(\mathrm{Span}[ \Sigma_{(\varphi,s)} (\mR\chi_D)])=2$.  This completes the proof.\qed
\end{proof}

Fig. \ref{propagations} illustrates how streaking artifacts are produced by the geometric structure of $D$. In Fig. \ref{propagations}, $D$ is given by  $D=\big\{\x\in\R^2:|\x|\leq1~\mbox{and }x_1, x_2\geq0\big\}$ so that  $\ssu(\chi_D)$ contains the line segments $L_{0,0}\cap \p D$ and $L_{\f{\pi}{2},0}\cap \p D$. Using arguments in the proof of Theorem \ref{th32}, we have
\begin{align*}
\dim(\mathrm{Span}[ \Sigma_{(\varphi,s)} (\mR\chi_D)])=2\q\mbox{if } (\varphi,s)= (0,0)~\mbox{or} ~(\pi/2,0).
\end{align*}
Hence, the lines $L_{0,0}$ and $L_{\f{\pi}{2},0}$ can be included in $\ssu(\mR^*\mI^{{\tiny\mbox{$-1$}}}((\mR\chi_D)^2)$. Fig. \ref{propagations} shows that the lines $L_{0,0}$ and $L_{\f{\pi}{2},0}$ are streaking artifacts.

\begin{figure}[ht]
\begin{center}
		\includegraphics[width=1\textwidth]{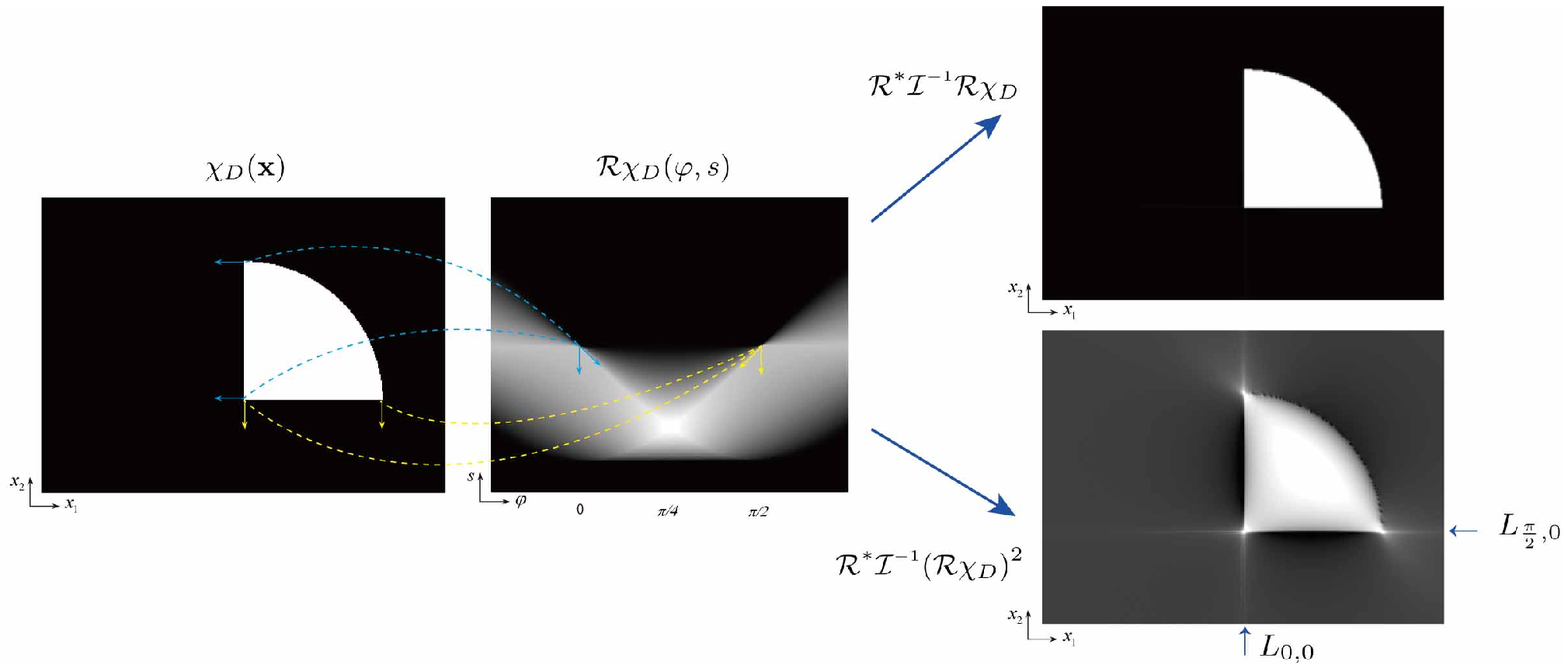}
				\caption{Illustration of streaking artifacts in $\mR^*\mI^{\tiny{-1}}((\mR\chi_D)^2)$. As described in Theorem \ref{th32}, streaking artifacts in $\mR^*\mI^{{\tiny\mbox{$-1$}}}((\mR\chi_D)^2)$ are produced only when $(\varphi,s)\in\ssu(\mR\chi_D)$ satisfies $\dim(\mathrm{Span}[ \Sigma_{(\varphi,s)} (\mR\chi_D)])=2$. In contrast, $\mR^*\mI^{{\tiny\mbox{$-1$}}}\mR\chi_D$ has no streaking artifact because of one to one correspondence between $\WF(\mR \chi_D)$ and $\WF(\mR^*\mI^{{\tiny\mbox{$-1$}}}\mR\chi_D)=\WF(\chi_D)$ even though $\dim(\mathrm{Span}[\Sigma_{(\varphi,s)} (\mR\chi_D)])=2$.}
		\label{propagations}
\end{center}
\end{figure}

Next, we restrict ourselves to the case where $D$ is simply connected. The following assertion is a direct consequence of Theorem 1.
\begin{theorem}\label{th31} Let $D\subseteq\R^2$ denote a metal region with the connected $C^2$ boundary $\p D$.  If $D$ is strictly convex, then the CT image $f_{\CT}$ does not have the streaking artifacts in the sense of wavefront set
\begin{align*}
\WF(f_{\CT})\subseteq\WF(f_{E_0}).
\end{align*}
\end{theorem}

\begin{proof} Recall that $\WF(\chi_D)=\big\{(\x,a\bth):\x\in\p D,~a\neq0,~\&~\bth\perp T_{\x}(\p D)\big\}$
where $T_{\x}(\p D)$ is the tangent space of $\p D$ at $\x$. Since $D$ is strictly convex,
\begin{align*}
L_{\varphi,\x\cdot\bth} \cap \p D =\{\x\}\q\mbox{for all } (\x,\bth)\in \WF(\chi_D).
\end{align*}
Since $\WF(\mR\chi_D)=\mC_{\mR}\circ\WF(\chi_D)$, we have
\begin{align*}
\dim(\mathrm{Span}[ \Sigma_{(\varphi,s)} (\mR\chi_D)])=1\q\mbox{for all } (\varphi,s)\in \ssu(\mR\chi_D).
\end{align*}
From Theorem \ref{th32}, $f_{\CT}$ does not have streaking artifacts. (See Fig. \ref{fig:streaking_illustration}.) \qed

\end{proof}
Theorem \ref{th31} implies that we have
\begin{align}
\WF(f_{\CT})\nsubseteq\WF(f_{E_0}).
\end{align}
only when if the metal region $D$ is not strictly convex.  In other words, the streaking artifacts in $f_{\CT}$ are related with the geometry of $D$, as shown in Fig. \ref{fig:streaking_illustration}. In the figure, we use the Shepp-Logan phantom as $f_{E_0}$ and the homogeneous metallic objects $\chi_D$ with the various geometries are added to illustrate the streaking artifacts in the reconstructed image $f_{\CT}$.

\begin{figure}[ht]
\begin{center}
		\includegraphics[width=1\textwidth]{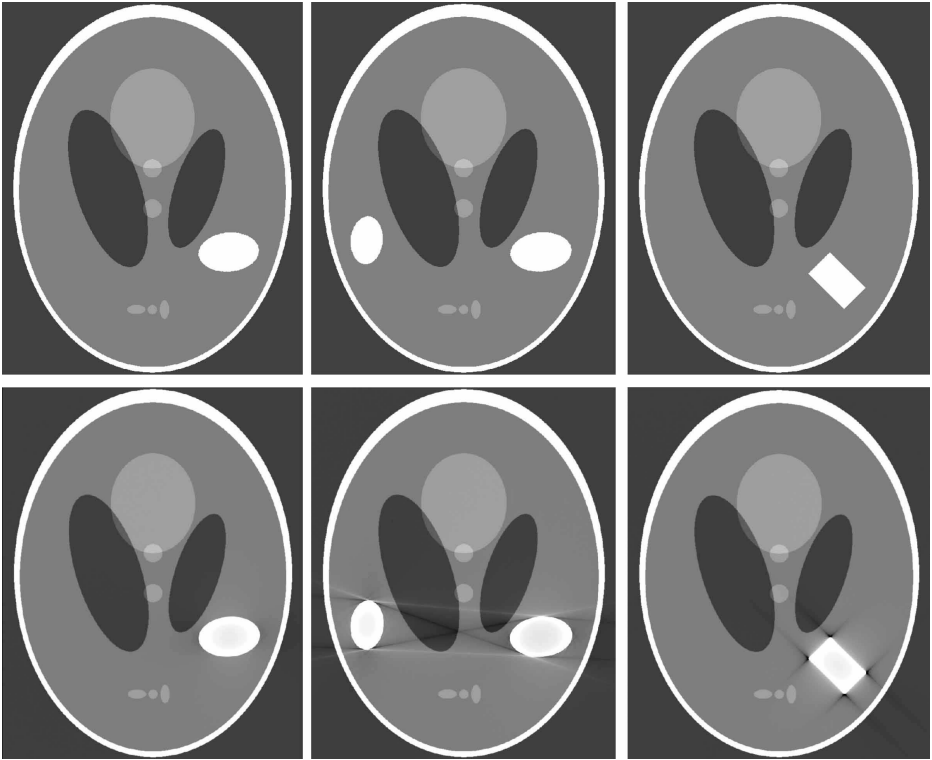}
		\caption{Illustration of streaking artifacts in the presence of metallic objects. First row and second row shows the Shepp-Logan phantom $f_{E_0}$ with various metallic objects and the corresponding reconstructed  $f_{\CT}$ \eref{fct} with the display window [-0.02 0.04], respectively.}\label{fig:streaking_illustration}
\end{center}
\end{figure}
According to Theorem \ref{th32} and Proposition \ref{proposition22},  metal streaking artifacts are due to the severe nonlinearity of the X-ray data $P$ with respect to the geometry of the metallic subject $D$. Metal streaking artifacts are included in the union of tangent space $T_{\x}(\p D)$, which is a tangent space $T_{\tilde\x}(\p D)$ of another point $\tilde\x\in \p D$. In the case of single metallic object with a strictly convex boundary, the reconstructed CT image has no streaking artifact.  If $D$ consists of two simply connected domains $D_1$ and $D_2$ with $C^2$ boundaries respectively, there are four different tangent lines touching both the metal domains $\p D_1$ and $\p D_2$.

\section{Other Causes of Streaking Artifacts}\label{othereffects}
In the previous section, we investigate the metal streaking artifacts caused by beam hardening effects due to poly chromatic X-ray sources.  The main difficulty in handling such artifacts comes from the nonlinear relationship between the projection $P$ and $f_E$.   We provide a rigorous analysis of the way  in which beam hardening effects generate wavefront-shaped streak artifacts that appear near materials such as metal.

Compton scattering  \cite{Glover1982}  can also cause streaking artifacts by altering the direction and energy of the X-ray beams. It  causes the X-ray data $P(\varphi,s)$ to deviate from the range of the Radon transform, and gives a nonlinear relation between $P$ and $f_E$. We consider the following simple model of Compton scattering using a monochromatic X-ray source \cite{DeMan1998,Glover1982}; let $f_{E_0}$ be a piece-wise constant function given by
\begin{align}\label{f_E0_piecewiseconst}
f_{E_0}=\sum_{n=0}^N\alpha_n\chi_{D_n}
\end{align}
where $D_n$'s $(n=1,2,\cdots,N)$ are subdomains of a strictly convex set $D_0$ containing a cross-sectional slice to be imaged and $\alpha_n$'s are positive constants.  Then we assume that the projection data $P$ is given as \cite{Glover1982}
\begin{align}\label{P_scatter}
P(\varphi,s)=-\ln\left(\exp\big\{-\mR f_{E_0}(\varphi,s)\big\}+c\chi_{\mQ}(\varphi,s)\right)\q(\forall (\varphi,s)\in (-\pi,\pi]\times \R),
\end{align}
where $c$ is a positive constant representing the scattered intensities approximately (See Remark \ref{rmk1} below) and $\mQ=\su(\mR\chi_{D_0})$.

\begin{remark}\label{rmk1} In 1982, Glover observed that the scattered radiations are relatively uniform for the round metallic materials \cite{Glover1982}.  Hence, we can approximate the scatter intensity as a constant and write the projection $P(\varphi,s)$ as \eref{P_scatter}.
\end{remark}

The following corollary explains that the nature of the scatter artifacts is similar to the beam hardening artifacts and  Fig. \ref{streaking_discussion} (b) shows these phenomena described in \cite{DeMan1998,Glover1982}.

\begin{corollary}\label{col1}  Let $f_{E_0}$ be a piece-wise constant function given by \eref{f_E0_piecewiseconst}.  Assume that the projection data $P(\varphi,s)$ is the function in \eref{P_scatter}. If a line $L_{\varphi,s}$ is a streaking artifact of $f_\CT$ in the sense of wavefront set \eref{artifact0}, then $(\varphi,s)$ satisfies
\begin{align}\label{streaking5}
\dim\left(\mathrm{Span}[\Sigma_{(\varphi,s)}(\mR\chi_{\cup_{n=1}^N D_n})]\right)=2
\end{align}
\end{corollary}
\begin{proof} The projection data $P(\varphi,s)$ in \eref{P_scatter} can be expressed as
\begin{align*}
P&=-\ln\left(\exp\{-\mR f_{E_0}\} (1+c\chi_{\mQ}\exp\{\mR f_{E_0}\}\right)\\
&=\mR f_{E_0}+\sum_{m=1}^{\infty}\f{(-1)^{m}}{m}  (c\chi_{\mQ})^m\left[\exp\{\mR f_{E_0}\}\right]^m\\
&=\mR f_{E_0}+\sum_{m=1}^{\infty}\f{(-1)^{m}}{m}  (c\chi_{\mQ})^m\left[\sum_{k=0}^\infty \f{1}{k!}\left\{\sum_{n=0}^N \alpha_n\mR \chi_{D_n}\right\}^k\right]^m
\end{align*}
Applying a similar argument in the proof of Theorem \ref{th32}, we complete the proof.
\end{proof}

Corollary \ref{col1} implies that the streaking artifacts due to Compton scattering can occur between bones and metals, even when $f_E=f_{E_0}$ (See Fig. \ref{streaking_discussion} (b).)

Another potential source of streaking artifacts is photon noise. For the simple and clear explanation of the streaking artifacts due to noise, we assume that the noise $N$ in X-ray intensity attenuation $I(\varphi,s)=\exp\big\{-P(\varphi,s)\big\}$ is given by
\begin{align}\label{P_noise0}
N(\varphi,s)=\sum_{k=1}^K c_k\delta(\varphi-\varphi_k,s-s_k),
\end{align}
where $\delta(\cdot,\cdot)$ is the Dirac delta function and $c_1,\cdots,c_K$ are positive constants which follow the Poisson probability distribution \cite{Guan1996}.  With this assumption, the projection data $P$ is given by
\begin{align}\label{P_noise}
P(\varphi,s)=-\ln\left(\exp\big\{-\mR f_{E_0}(\varphi,s)\big\}+\sum_{k=1}^K c_k\delta(\varphi-\varphi_k,s-s_k)\right)
\end{align}
\begin{corollary}\label{col2}  Assume that the projection data $P(\varphi,s)$ is given as  \eref{P_noise}. Then, streaking artifacts of $f_\CT$ are included in the union of lines $\cup_{k=1}^KL_{\varphi_k,s_k}$.
\end{corollary}
\begin{proof} The projection data $P(\varphi,s)$ in \eref{P_noise} can be expressed  as
\begin{align*}
P(\varphi,s)&=\mR f_{E_0}(\varphi,s)-\ln\left(1+\sum_{k=1}^K c_k\delta(\varphi-\varphi_k,s-s_k)\exp\{\mR f_{E_0}(\varphi_k,s_k)\}\right)\\
&=\mR f_{E_0}(\varphi,s)+\sum_{m=1}^{\infty}\f{(-1)^m}{m}\left[\sum_{k=1}^Kc_k\delta(\varphi-\varphi_k,s-s_k)\exp\left\{\mR f_{E_0}(\varphi_k,s_k)\right\}\right]^m.
\end{align*}
Hence, we have
\begin{align*}
\big\{(\varphi,s)\in (-\pi,\pi]\times\R: \dim\left(\mathrm{Span} [ \Sigma_{(\varphi,s)}(P)]\right)=2\big\}=\big\{(\varphi_k,s_k): k=1,\cdots, K\big\}
\end{align*}
and the result follows from a similar argument in the proof of Theorem \ref{th32}. \end{proof}

Fig. \ref{streaking_discussion} (c) shows the streaking artifacts caused by noises in the form of \eref{P_noise}.

\begin{figure}[ht]
\begin{center}
		\includegraphics[width=1\textwidth]{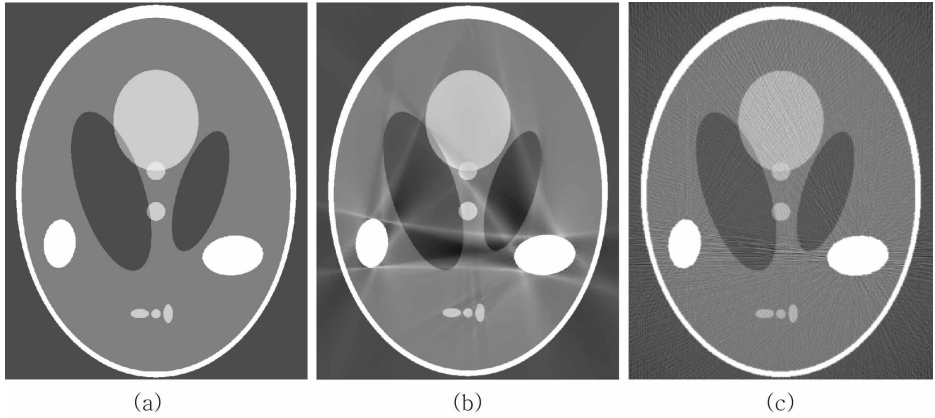}
		\caption{Illustration of streaking artifacts in the presence of metallic objects. First, second and third figure shows reference  image and reconstructed image due to scattering and noise with the display window [-0.02 0.04], respectively.}\label{streaking_discussion}
\end{center}
\end{figure}

One of other potential sources of metal artifacts is photon starvation, which occurs when insufficient (possibly zero) photons reach the detector as the X-ray beam passes through a metallic object.  This photon starvation generates a similar effect to the noise effect in the metal region \cite{Barrett2004}, and the noise contribution is greater from projections that pass through metallic objects than from those that do not.  Consequently, these noise effects lead to serious streaking artifacts in the reconstructed image at the points where the beam passes through metallic objects.   Numerical and experimental results of streaking artifacts due to photon starvation are well depicted in works by \cite{Barrett2004,Mori2013,vZabic2013}.  Streaking artifacts due to photon starvation are verified to be prominent along lines passing through multiple or strong  objects with high attenuation coefficients (e.g.  metal).  Finally, we should mention that streaking artifacts caused by scattering, noise, and photon starvation can be amplified when they are associated with metallic objects \cite{Barrett2004,DeMan1998,Mori2013,vZabic2013} (See Fig. \ref{real_image}).

To validate our main results in the clinical CT case, we include CT image of one author's teeth and mandible in Fig. \ref{real_image}.  As shown in this figure, most streaking artifacts occur along the tangent line of boundary of metallic objects.  Besides, due to scattering and/or noise effect, the streaking artifacts can occur between the metallic object and the bone, as described in the corollaries of this section.

\begin{figure}[ht]
\begin{center}
		\includegraphics[width=0.7\textwidth]{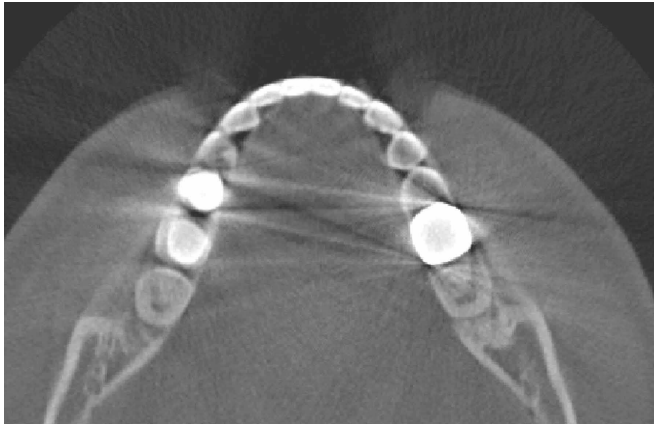}
		\caption{Illustration of streaking artifacts in clinical CT image.  The presence of the streaking artifacts in the clinical CT image mostly agrees with the description in our main theorems and their corollaries.}\label{real_image}
\end{center}
\end{figure}

\section{Remark on MAR and Discussion}
The growing number of patients with metallic implants has increased the importance of metal artifact reduction (MAR) in the application of clinical CT for diagnostic imaging for craniomaxillofacial and orthopedic surgery, in oncology; and when dental implants and prosthodontics are present.  The goal of MAR is that for a given projection data $P\in\msE'((-\pi,\pi]\times\R)$, we aim to find $P^{\natural}\in\msE'((-\pi,\pi]\times\R)$ such that
\begin{enumerate}
\item[~~~1.] $P^{\natural}\in\mbox{Range space}=\big\{\mR f^{\natural}:f^{\natural}\in\msE'(\R^2)\big\}$,
\item[~~~2.] $\|P-P^{\natural}\|$ is small with a suitable norm.
\end{enumerate}

Lewitt and Bates \cite{Lewitt1978} introduced the first method for reducing metal artifacts in the late 1970s.  Most existing methods incorporate some variation or combination of interpolation methods and iterative reconstruction methods with some regularizations \cite{Abdoli2010,Mouton2012}.

To explain the interpolation methods, we denote $D_{\varphi}=\su(\mR\chi_D(\varphi,\cdot))$ for each $\varphi\in(-\pi,\pi]$.  Then $P(\varphi,s)$ can be expressed as
\begin{align*}
P(\varphi,s)=P_{\mathrm{int}}(\varphi,s)\chi_{D_{\varphi}}(s)+P_{\mathrm{ext}}(\varphi,s)\chi_{D_{\varphi}^c}(s),
\end{align*}
where $P_{\mathrm{int}}$ is the projection data in $D_{\varphi}$ and $P_{\mathrm{ext}}$ outside of $D_{\varphi}$.  Interpolation methods aim to obtain $P^{\natural}$ by recovering $P_{\mathrm{int}}^{\natural}$ from the sinogram $P$ which is affected by metallic subjects under the assumption that
\begin{align*}
P^{\natural}=P_{\mathrm{ext}}~~~\mbox{on}~~~\p D_{\varphi}.
\end{align*}
The restoring techniques for $P_{\mathrm{int}}^{\natural}$ include linear interpolation \cite{DeMan2001,Kalender1987,Lewitt1978}, higher-order polynomial interpolation \cite{Abdoli2010,Bazalova2007,Roeske2003}, wavelets \cite{Zhao2001,Zhao2000}, Fourier transform \cite{Kratz2008}, tissue-class models \cite{Bal2006}, normalized interpolation methods \cite{Meyer2010}, and total variation \cite{Shen2002} or fractional-order inpainting methods \cite{Zhang2011}.  However, these interpolation methods suffer from inherent limitations; the mismatch of the restored $P^{\natural}$ with the range of the Radon transform \cite{Meyer2010,Muller2009}. According to our characterization, the inaccurate recovery of $P_{\mathrm{int}}^{\natural}$ using the interpolation near $\p D_{\varphi}$ will cause the additional singularities in the data, which can lead to additional streaking artifacts in the reconstructed CT image.

Park \etal \cite{Park2013} have proposed a PDE-based MAR method to recover background data hidden by metal's X-ray data from the observation that the Laplacian of projection data can capture the dental information influenced by the data from metal.  This method is based on the assumption that the given data $P$ can be decomposed into the two parts:
\begin{align}\label{P_decompose}
P=u^{\natural}+w
\end{align}
where $w$ denotes the projection data from metal which contains the discrepancy, and $u^{\natural}$ is the background data in the absence of metal which is the Radon transform of the background image $f_{\mathrm{b}}^{\natural}$ \cite{Park2013}. Then $u^{\natural}$ can be obtained by solving the following Poisson's equation:
\begin{align}\label{MAR_Poisson}
\left\{\begin{array}{cl}
\na^2 u_{\mathrm{int}}^{\natural}=\zeta(\na^2 P_{\mathrm{int}})~&~\mbox{in}~D_{\varphi}\\
u_{\mathrm{int}}^{\natural}=P_{\mathrm{ext}}~&~\mbox{in}~\p D_{\varphi}\end{array}\right.
\end{align}
where $\zeta$ is an operator properly chosen to keep $\ssu(P_{\mathrm{int}})$ in $D_{\varphi}$ \cite{Park2013}.  Since we can recover $\ssu(u^{\natural})=\ssu(\mR f_{\mathrm{b}}^{\natural})\subseteq\ssu (P)$ in $D_{\varphi}$ using \eref{MAR_Poisson}, the corresponding $\ssu (f_{\mathrm{b}}^{\natural})$ in the reconstructed CT image can be restored based on the one to one correspondence between $\WF(f_{\mathrm{b}}^{\natural})$ and $\WF(u^{\natural})=\WF(\mR f_{\mathrm{b}}^{\natural})$.  However, this method has some limitations that it may not work well in the presence of thick and/or multiple metallic objects.

Iterative reconstruction methods generally employ the data fitting method to reduce the streaking artifacts caused by the discrepancy between the given data $P$ and $\mR f^{\natural}$, which can be achieved from the following minimization problem:
\begin{align}\label{MAR_iterative}
f^{\natural}=\arg\min~\mathrm{Reg}(f^{\natural})+\lambda\mathrm{Fit}(\mR f^{\natural}-P).
\end{align}
Here, $\mathrm{Reg}(f^{\natural})$ denotes the regularization term which enforces the regularity of $f^{\natural}$, $\mathrm{Fit}(\mR f^{\natural}-P)$ is the fidelity term which forces the mismatch between $\mR f^{\natural}$ and $P$ to be small, and $\lambda>0$ is the regularization parameter.  The iterative methods include the statistical data-fitting methods such as maximum-likelihood for transmission \cite{DeMan2000}, expectation maximization \cite{Shepp1982,Wang1996}, and iterative maximum-likelihood polychromatic algorithm for CT \cite{DeMan2001}.  These data-fitting methods can alleviate the streaking artifacts by reducing the mismatch, even though these methods are known to have a disadvantage of huge computational costs \cite{J.Choi2011,Park2013}.

The spatial prior information of $f^{\natural}$ can be imposed as a regularization term to obtain $f^{\natural}$ with the desired property \cite{Yan2011}.  For example, the total variation $\|f^{\natural}\|_{\TV}$ can be used to reduce the streaking artifacts:
\begin{align}
f^{\natural}=\arg\min~\|f^{\natural}\|_{\TV}+\lambda\mathrm{Fit}(\mR f^{\natural}-P).
\end{align}
Based on the theory of compressed sensing (CS) that the sparse signals can be recovered exactly and stably from a partial measurement data \cite{Cand`es2006,Donoho2006}, we can obtain $f^{\sharp}$ stably using only a part of the projection data \cite{Yan2011}.  However, due to the inherent nature of the total variation, this model may not be effective to the clinical CT data.  Indeed, the inappropriate choice of $\lambda$ will lead to the lack of realistic variations, which may hamper the clinical and/or scientific applications.

Recently, Choi \etal \cite{J.Choi2011} propose an MAR method to reconstruct the artifact-free metal part image $f_{\mathrm{m}}^{\natural}:=f^{\natural}-f_{\mathrm{b}}^{\natural}$ after recovering $u^{\natural}$ by the linear interpolation with taking advantage of the spatial sparsity of $f_{\mathrm{m}}^{\natural}$:
\begin{align}\label{MAR_CS}
\begin{array}{ll}
\mbox{minimize}~&~\|f_{\mathrm{m}}^{\natural}\|_1\\
\mbox{subject to}~&~\mathrm{Fit}(\mR f_{\mathrm{m}}^{\natural}-w)\leq\eps.
\end{array}
\end{align}
From the assumption that $f_{\mathrm{m}}^{\natural}$ occupies only a small portion of the image domain \cite{J.Choi2011}, the CS theory will guarantee the stable and robust recovery of $f_{\mathrm{m}}^{\natural}$ using a part of the projection data.  However, since the linear interpolation is used to obtain the background data $u^{\natural}$, there could remain the streaking artifacts in the reconstructed image.

Despite the rapid advances in CT technologies and various works seeking to reduce metal artifacts, metal streaking artifacts continue to pose difficulties, and the development of suitable reduction methods remains challenging. Previously, the wavefront set has been used to characterize artifacts due to the limited angle tomography \cite{J.Frikel2013,Quinto1993} as well as cone beam local tomography \cite{Katsevich1999,Katsevich2006}. Moreover, there have been many studies regarding the wavefront set of the Radon transform \cite{Faridani2003,Finch2003,Greenleaf1989,Quinto1980,Quinto2007,Ramm1996,Ramm1993}. To our surprise there has been no mathematical analysis of metal artifacts in terms of wavefront set.  In this paper, we report for the first time that the wavefront set can also be used to explain the characterization of metal artifacts.  As explained in our main result (Theorem \ref{th32}), metal streaking artifacts are produced only when the wavefront set of $\mR\chi_D$ does not contain the wavefront set of $(\mR\chi_D)^2$.  These characterizations lead us to explain that the streaking artifacts arise mainly from the geometry of the boundary of the metallic objects. Besides, we can provide some mathematical analysis on other factors that also cause streaking artifacts (Section \ref{othereffects}.)

These theoretical studies will be helpful for the development of MAR methods to reduce the artifacts effectively; based on our main result, the structure of streaking artifacts can be extracted from the reconstructed CT image provided that the geometry information of the teeth and mandible is given.  By doing so, we will be able to reduce the streaking artifacts effectively.  Unfortunately, it remains a future work to obtain the exact geometry of the mandible because in general we obtain the geometry of the teeth and mandible in empirical and statistical ways.  In addition, even though we have provided the qualitative nature of the metal streaking artifacts, the quantitative analysis on the metal artifacts will be required so as to reduce the artifacts based on our characterizations.  Nevertheless, it should be noted that this paper is the first approach to provide a mathematical characterization of metal artifacts in CT.  Further researches on metal artifacts will be needed to achieve more accurate reconstruction of CT image as well as more efficient MAR.

\begin{acknowledgements}
This work was supported by the National Research Foundation of Korea (NRF) grant funded by the Korean Government (MEST) (No.2011-0028868,2012R1A2A1A03
670512).
\end{acknowledgements}

\bibliographystyle{spmpsci}      


\end{document}